\documentclass{amsart}

\usepackage{amsmath,amsthm,amssymb}
\usepackage[latin1]{inputenc}
\usepackage{subfigure}
\usepackage[usenames,dvipsnames,svgnames,table]{xcolor}
\usepackage{young}
\usepackage{cite}
\usepackage{array}
\usepackage{tikz}
\usepackage{url}
\input xy
\xyoption{all}

\input xy \xyoption{all}

\newtheorem{theorem}{Theorem}[section]
\newtheorem{lemma}[theorem]{Lemma}

\newtheorem{corollary}[theorem]{Corollary}

\newtheorem*{MainResult}{Theorem~\ref{quarterturn}}

\theoremstyle{definition}
\newtheorem{definition}[theorem]{Definition}
\newtheorem{example}[theorem]{Example}

\theoremstyle{remark}

\begin{document}

\author{Kevin Dilks}

\title{Quarter-Turn Baxter Permutations}

\date{\today}


\begin{abstract}
Baxter permutations are known to be in bijection with a wide number of combinatorial objects. Previously, it was shown that each of these objects had a natural involution which was carried equivariantly by the known bijections, and the number of objects fixed under involution was given by Stembridge's ``q=-1'' phenomenon. In this paper, we consider the order 4 action of a quarter-turn rotation of a Baxter permutation matrix, refining the half-turn rotation previously studied. Using the method of generating trees, we show that the number of Baxter permutations fixed under quarter-turn rotation has a very nice enumeration, which suggests the existence of a combinatorial bijection.
\end{abstract}

\maketitle



\tableofcontents

\section{Background}
\label{intro}

Baxter permutations are a well-studied class of permutations, which have a number of symmetries and nice properties associated to them.

\begin{definition}
We say that a \emph{Baxter permutation} is a permutation that avoids the patterns 3-14-2 and 2-41-3, where an occurrence of the pattern 3-14-2 in a permutation $w=w_1\ldots w_n$ means there exists a quadruple of indices $\{i,j,j+1,k\}$ with $i<j<j+1<k$ and $w_j< w_k< w_i< w_{j+1}$ (and similarly for 2-14-3)\footnote{Such patterns with prescribed adjacencies are sometimes called {\textit{vincular} patterns}.}.
\end{definition}

For $n=4$, there are $B(4)= 22$ Baxter permutations in $\mathfrak{S}_4$, with the only excluded ones being 2413 and 3142. The general formula for the number of Baxter permutations of length $n$ is given by

\[ B(n):=\sum_{k=0}^{n-1}\frac{\binom{n+1}{k}\binom{n+1}{k+1}\binom{n+1}{k+2}}{\binom{n+1}{1}\binom{n+1}{2}}, \]

and was originally proven by Chung, Graham, Hoggatt, and Kleiman~\cite{NumBax}.

It is easy to see from the definition that Baxter permutations will be closed under two natural involutions. One of them reverses the order of a word ($w=w_1\ldots w_n\mapsto w_n\ldots w_1$), and the other reverses the order of the labels ($w=w_1\ldots w_n \mapsto (n+1-w_1)\ldots (n+1-w_n)$). These correspond to reflecting a permutation matrix horizontally and vertically (respectively). A slightly less obvious fact is that Baxter permutations will be closed under taking inverses, which corresponds to reflecting the permutation matrix across a diagonal line. This means that Baxter permutations are closed under the full dihedral action of the square.

It is clear that the first two involutions individually will never have any fixed points for $n>1$. 

The author has previously shown that the combination of the first two involutions (correspond to a half-turn of the permutation matrix) is carried equivariantly to a natural rotation on other combinatorial objects, and the enumeration of fixed points is an instance of the ''$q=-1$ phenomenon``. \cite{Dilks} 

Baxter permutations fixed under reflection across the diagonal correspond to self-involutive Baxter permutations, and these have previously been considered. The enumerative formula for the number of fixed-point free self-involutive Baxter permutations of length $2n$ has the surprisingly simple closed formula $b_n=\frac{3\cdot 2^{n-1}}{(n+1)(n+2)}\binom{2n}{n}$ through a bijection to planar maps \cite{fpfsi}. Later, Fusy extended this method to give a combinatorial proof of the enumeration for $b_n$, as well as a closed-form multivariate enumeration for all self-involutive Baxter permutations \cite{Fusy}.

The last remaining conjugacy class of dihedral actions on Baxter permutations is the one corresponding to $90^{\circ}$ rotation, which we now consider.

Our main result gives an enumeration of the number of Baxter permutations fixed under this quarter-turn rotation.

\begin{theorem}
\label{quarterturn}
The number of Baxter permutations of length $n$ fixed under $90^{\circ}$ rotation of its permutation matrix is $2^mC_m$ (where $C_m$ is the Catalan number) if $n=4m+1$, and zero otherwise.
\end{theorem}

We will prove this using the method of generating trees. In Section~\ref{sec:GenTreePerm}, we will recall some background on generating trees for families of permutations. In Section~\ref{sec:GenTreeBax}, we will recall the results of Chung, Graham, Hoggatt, and Kleiman used to give the original enumeration of Baxter permutations. ~\cite{NumBax}. In Section~\ref{sec:GenTreeHalf}, we will extend these results to describe the generating tree for Baxter permutations fixed under a $180^{\circ}$ rotation. Then in Section~\ref{sec:GenTreeQuarter}, we will further extend this to describe the generating tree for Baxter permutations fixed under a $90^{\circ}$ rotation of the permutation matrix, and prove our main result.

\section{Generating trees for permutations}
\label{sec:GenTreePerm}

Say we have a family of permutations that is closed under removing the largest entry. Then every permutation in the family of length $n$ arises uniquely from taking a permutation of length $n-1$ in the family, and inserting the letter $n$ into an admissible position. 

\begin{definition}
We say that the \emph{generating tree} of a family of permutations closed under removing the largest entry is the tree whose nodes are the permutations in the family, and the parent of each node is the permutation obtained by removing the largest entry.
\end{definition}

\noindent In many cases, one can obtain enumeration results by analyzing this tree.

The family of permutations avoiding some set of classical patterns in always closed under taking any subword (in particular, removing the largest element), so we can construct a generating tree.

\begin{example}

Consider the set of permutations that avoid the classical pattern $231$. It is not hard to see given a $231$ avoiding permutation, the only place one can insert a new largest label into a permutation and still avoid the pattern $231$ is immediately to the left of a left-to-right maxima or at the end of the permutation. We say that $w_i$ is a left-to-right maxima of $w$ if $w_i>w_k$ for all $k<i$. So the number of children a permutation has in the generating tree depends only on this statistic.

Furthermore, inserting a new largest label has a predictable effect on the number of left-to-right maxima of the resulting permutation. If a permutation has $k+1$ left-to-right maxima, then it will have $k+1$ children with $2,3,\ldots, k+2$ left-to-right maxima.

Thus, an abstract tree with nodes labelled by integers that has root 1 and the property that every node $k+1$ has children labelled $2,3,\ldots k+2$ will be isomorphic to the generating tree for $231$ avoiding permutations. This tree is known as the \emph{Catalan tree}~\cite{WestGenTree}, and is known to have rank sizes corresponding to the Catalan numbers. See Figures~\ref{231fig} and \ref{cattree} for the generating tree of $231$ avoiding permutations, and the associated Catalan tree.

Similarly, we could consider permutations that avoided the classical pattern $132$. In this case, the places where we could insert a new largest label are immediately to the right of a right-to-left maxima (where we say that $w_i$ is a right-to-left maxima of $w$ if $w_i>w_k$ for all $k>i$) or at the beginning of the permutation. We again have the same predictable effect on number of right-to-left maxima by inserting a new largest label into a fixed permutation in all possible ways, and we again get a generating tree isomorphic to the Catalan tree.

\end{example}

\begin{figure}

\resizebox{\textwidth}{!}{
\begin{tikzpicture}
\tikzstyle{level 1}=[sibling distance=63mm]
\tikzstyle{level 2}=[sibling distance=27mm]
\tikzstyle{level 3}=[sibling distance=8mm]
\node  (z){1}
    child {node  (a) {21}
        child {node  (b) {321}
            child {node {\small 4321}}
			child {node {\small 3214}}
        }
        child {node  (d) {213}
            child {node {\small 4213}}
			child {node {\small 2143}}
			child {node {\small 2134}}
        }
    }
    child {node  (e) {12}
        child {node  (f) {312}
            child {node {\small 4312}}
			child {node {\small 3124}}
        }
        child {node  (g) {132}
            child {node {\small 4132}}
			child {node {\small 1432}}
			child {node {\small 1324}}
        }
        child {node  (h) {123}
            child {node {\small 4123}}
			child {node {\small 1423}}
			child {node {\small 1243}}
			child {node {\small 1234}}            
        }
    };    
	\end{tikzpicture}
	}
\caption{The beginning of the generating tree for $231$-avoiding permutations.}
\label{231fig}
\end{figure}
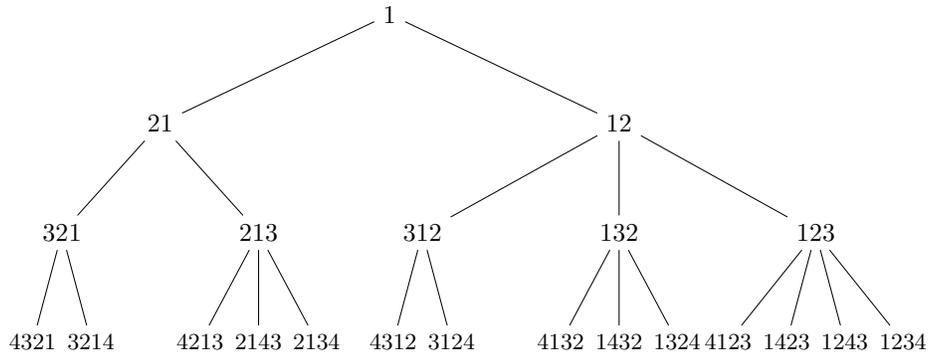

\begin{figure}
\resizebox{\textwidth}{!}{
\begin{tikzpicture}
\tikzstyle{level 1}=[sibling distance=59mm]
\tikzstyle{level 2}=[sibling distance=21mm]
\tikzstyle{level 3}=[sibling distance=5mm]
\node  (z){2}
    child {node  (a) {2}
        child {node  (b) {2}
            child {node {2}}
			child {node {3}}
        }
        child {node  (d) {3}
            child {node {2}}
			child {node {3}}
			child {node {4}}
        }
    }
    child {node  (e) {3}
        child {node  (f) {2}
            child {node {2}}
			child {node {3}}
        }
        child {node  (g) {3}
            child {node {2}}
			child {node {3}}
			child {node {4}}
        }
        child {node  (h) {4}
            child {node {2}}
			child {node {3}}
			child {node {4}}
			child {node {5}}            
        }
    };  
	\end{tikzpicture}
	}
\caption{The beginning of the Catalan tree}
\label{cattree}
\end{figure}
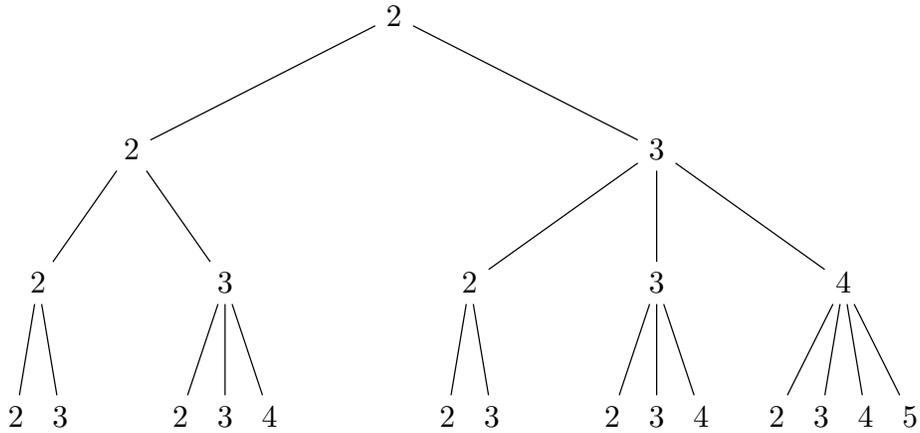

\section{Generating tree for Baxter permutations}
\label{sec:GenTreeBax}

Baxter permutations are given by a vincular pattern, where we have adjacency issues to consider, so it is not immediately obvious that they are closed under removing the largest label.

\begin{lemma}
\label{remove}
If $w$ is a Baxter permutation, and we remove its largest label, then the result is still a Baxter permutation
\end{lemma}

\begin{proof}
Say $w=w_1\ldots w_n$ is a Baxter permutation, and we remove $w_i=n$ to get $\bar{w}=w_1\ldots\hat{w_i}\ldots w_n$. If $\bar{w}$ is not a Baxter permutation, then WLOG say there is an instance of $2-41-3$. That means there is a subsequence $1\leq i_1<i_2<i_3<i_4\leq n$ with $i_1,i_2,i_3,i_4\neq i$, $w_{i_3}<w_{i_1}<w_{i_4}<w_{i_2}$, and $i_2$ is adjacent to $i_3$ in $\hat{w}$. The only way $i_2$ can be adjacent to $i_3$ in $\hat{w}$ is if $i_2+1=i_3$, or if $i_3+2=i+1=i_2$. In the first case, the subsequence $i_1,i_2,i_3,i_4$ would be an instance of $2-41-3$ in $w$, a contradiction of our assumption. In the second case, the subsequence $i_1,i,i_3,i_4$ would be an instance of $2-41-3$ in $w$, again a contradiction.
\end{proof}

Therefore, every Baxter permutation of length $n$ uniquely arises from taking a Baxter permutation of length $n-1$ and inserting $n$ into an admissible position. 

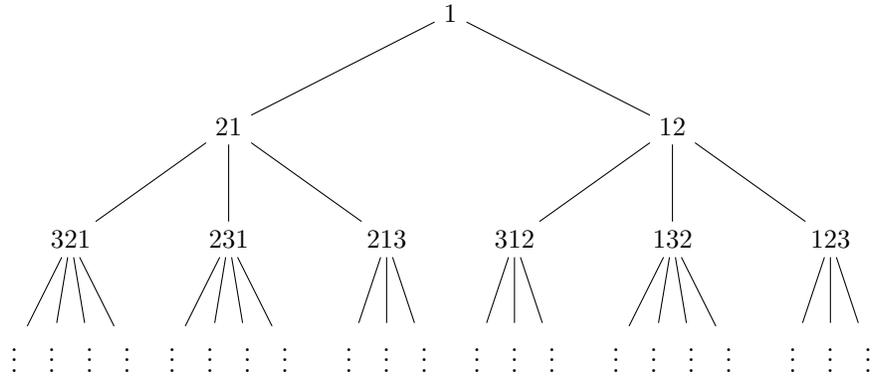
\begin{figure}
\begin{tikzpicture}
\tikzstyle{level 1}=[sibling distance=59mm]
\tikzstyle{level 2}=[sibling distance=21mm]
\tikzstyle{level 3}=[sibling distance=5mm]
\node  (z){1}
    child {node  (a) {21}
        child {node  (b) {321}
            child {node {$\vdots$}}
			child {node {$\vdots$}}
			child {node {$\vdots$}}
			child {node {$\vdots$}}
        }
        child {node  (c) {231}
            child {node {$\vdots$}}
			child {node {$\vdots$}}
			child {node {$\vdots$}}
			child {node {$\vdots$}}
        }
        child {node  (d) {213}
            child {node {$\vdots$}}
			child {node {$\vdots$}}
			child {node {$\vdots$}}
        }
    }
    child {node  (e) {12}
        child {node  (f) {312}
            child {node {$\vdots$}}
			child {node {$\vdots$}}
			child {node {$\vdots$}}
        }
        child {node  (g) {132}
            child {node {$\vdots$}}
			child {node {$\vdots$}}
			child {node {$\vdots$}}
			child {node {$\vdots$}}
        }
        child {node  (h) {123}
            child {node {$\vdots$}}
			child {node {$\vdots$}}
			child {node {$\vdots$}}
        }
    };
    
	\end{tikzpicture}

\caption{The beginning of the generating tree for Baxter permutations}
\end{figure}

Chung, Graham, Hoggatt, and Kleiman \cite{NumBax} studied this generating tree to come up with their enumerative result. They showed that the admissible places where we can insert a new largest label into a Baxter permutation are immediately to the left of a left-to-right maxima, and immediately to the right of a left-to-right maxima.

The resulting Baxter permutation will also have a predictable number of left-to-right and right-to-left maxima. Say $w$ has left-to-right maxima $x_1<x_2,\ldots<x_i=n$ and right-to-left maxima $n=y_j>y_{j-1}>\ldots>y_1$. If we insert $n+1$ to the left of $x_k$, the resulting permutation will have left-to-right maxima $x_1<\ldots <x_{k-1}<n+1$, and right-to-left maxima $n+1>n=y_j>\ldots>y_1$. If we insert $n+1$ to the right of $y_k$, the resulting permutation will have left-to-right maxima $x_1<x_2<\ldots x_i=n<n+1$, and right-to-left maxima $n+1>y_{k-1}>\ldots >y_1$.

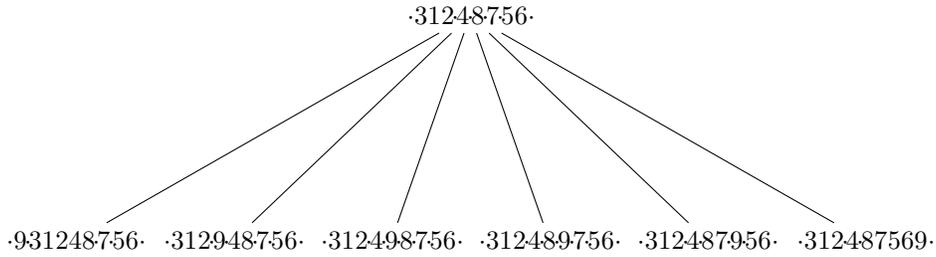
\begin{figure}
\begin{tikzpicture}
\tikzstyle{level 1}=[sibling distance=21mm]
\node (A) {$ \cdot  312 \!\! \cdot \!\! 4 \!\! \cdot \!\! 8 \!\! \cdot \!\! 7 \!\! \cdot \!\! 56  \cdot  $}
 [level distance=30mm]
 child {node {$ \cdot  9 \!\! \cdot \!\! 31248 \!\! \cdot \!\! 7 \!\! \cdot \!\! 56 \cdot $}}
 child {node {$  \cdot  312 \!\! \cdot \!\! 9\!\! \cdot \!\!48\!\! \cdot \!\!7\!\! \cdot \!\!56  \cdot  $}}
 child {node {$  \cdot  312 \!\! \cdot \!\! 4 \!\! \cdot \!\! 9\!\! \cdot \!\!8\!\! \cdot \!\!7\!\! \cdot \!\!56  \cdot  $}}
 child {node {$  \cdot  312 \!\! \cdot \!\! 4 \!\! \cdot \!\! 8 \!\! \cdot \!\! 9 \!\! \cdot \!\! 7 \!\! \cdot \!\! 56 \cdot  $}}
 child {node {$ \cdot 312 \!\! \cdot \!\! 4 \!\! \cdot \!\! 87 \!\! \cdot \!\! 9 \!\! \cdot \!\! 56 \cdot  $}}
 child {node {$  \cdot  312 \!\! \cdot \!\! 4 \!\! \cdot \!\! 87569 \cdot $}}
;
\end{tikzpicture}
\caption{Branching of generating tree for Baxter permutations at $w=31248756$, with insertion points marked.}
\end{figure}

This means that the number of children a given Baxter permutation has (and how many children those children will have, and so on) is entirely encoded by the number $i$ of left-to-right maxima, and the number $j$ of right-to-left maxima. Thus, the tree with root $(1,1)$, and the property that every node $(i,j)$ has children $(1,j+1), (2,j+1),\ldots (i,j+1),(i+1,j), (i+1,j-1), \ldots (i+1,1)$ will be isomorphic to the generating tree for Baxter permutations.

\begin{figure}
\begin{tikzpicture}
\tikzstyle{level 1}=[sibling distance=16mm]
\node (A) {$(i,j)$}
 [level distance=30mm]
 child {node {$(1,j+1)$}}
 child {node {$(2,j+1)$}}
 child {node {$\ldots$}}
 child {node {$(i,j+1)$}}
 child {node {$(i+1,j)$}}
 child {node {$(i+1,j-1)$}}
 child {node {$\ldots$}}
 child {node {$(i+1,1)$}}
;
\end{tikzpicture}
\caption{Rule for generating tree isomorphic to Baxter permutations}
\end{figure}
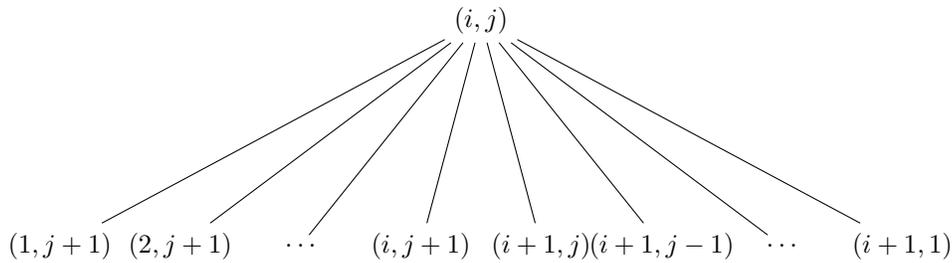

\begin{figure}
\begin{tikzpicture}
\tikzstyle{level 1}=[sibling distance=59mm]
\tikzstyle{level 2}=[sibling distance=21mm]
\tikzstyle{level 3}=[sibling distance=5mm]
\node  (z){$(1,1)$}
    child {node  (a) {$(1,2)$}
        child {node  (b) {$(1,3)$}
            child {node {$\vdots$}}
			child {node {$\vdots$}}
			child {node {$\vdots$}}
			child {node {$\vdots$}}
        }
        child {node  (c) {$(2,2)$}
            child {node {$\vdots$}}
			child {node {$\vdots$}}
			child {node {$\vdots$}}
			child {node {$\vdots$}}
        }
        child {node  (d) {$(2,1)$}
            child {node {$\vdots$}}
			child {node {$\vdots$}}
			child {node {$\vdots$}}
        }
    }
    child {node  (e) {$(2,1)$}
        child {node  (f) {$(1,2)$}
            child {node {$\vdots$}}
			child {node {$\vdots$}}
			child {node {$\vdots$}}
        }
        child {node  (g) {$(2,2)$}
            child {node {$\vdots$}}
			child {node {$\vdots$}}
			child {node {$\vdots$}}
			child {node {$\vdots$}}
        }
        child {node  (h) {$(3,1)$}
            child {node {$\vdots$}}
			child {node {$\vdots$}}
			child {node {$\vdots$}}
        }
    };
    
	\end{tikzpicture}

\caption{The beginning of the generating tree isomorphic to the Baxter permutation generating tree.}
\end{figure}
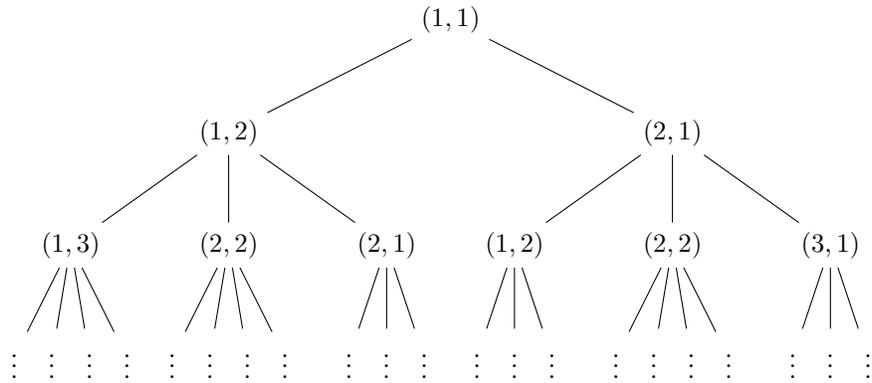

One interesting consequence of the generating tree approach gives a non-obvious relationship between two permutation statistics on Baxter permutations.

\begin{corollary}
\label{desisinvdes}
Baxter permutations have the same number of descents as inverse descents.

\end{corollary}

\begin{proof}
If $n$ is inserted to the left of a left-to-right maximum, then either $n+1$ is being added to the front of the word, or it is being inserted into an ascent. In either case, the resulting permutation will have one more descent than the original one. But since $n$ is always the rightmost left-to-right maximum, $n+1$ is being inserted to the left of $n$, so we are also creating one new inverse descent. Similarly, $n+1$ being inserted to the right of a right-to-left maximum creates no new descents nor inverse descents. Since the act of inserting a new largest label preserves the difference between number of descents and number of inverse descents, and the base of our generating tree has the same number of descents as inverse descents, all permutations in the generating tree have the same number of descents as inverse descents.
\end{proof}

\section{Generating tree for Baxter permutations fixed under $180^{\circ}$ rotation}
\label{sec:GenTreeHalf}

Now, let us consider the generating tree of all Baxter permutations fixed under $180^{\circ}$ rotation. 

A permutation $w$ of length $n$ being fixed under $180^{\circ}$ rotation means that if $w_i=j$, then $w_{n+1-i}=n+1-j$. By the same logic of Lemma~\ref{remove}, we can see that we can remove 1 from a Baxter permutation (and decrease all remaining labels by one) and still be a Baxter permutation. 

Combining these two things, we can see that if we remove $n$ and 1 (and then decrease all the labels by 1) from a Baxter permutation fixed under $180^{\circ}$ rotation, then we will still have a Baxter permutation fixed under $180^{\circ}$ rotation. So again, we can construct a generating tree. 

Note that in this case, we are removing two entries at a time, so we will have separate generating trees for when $n$ is even and when $n$ is odd. We will use the convention that the generating tree for $n$ even has the empty permutation $\emptyset$ of length 0 as its root, with children $12$ and $21$.

\begin{figure}
\begin{tikzpicture}
\tikzstyle{level 1}=[sibling distance=59mm]
\tikzstyle{level 2}=[sibling distance=12mm]
\tikzstyle{level 3}=[sibling distance=5mm]
\node  (z){1}
    child {node  (a) {321}
        child {node  (b) {54321}}
        child {node (c) {45312}}
        child {node (d) {41352}}
        child {node (e) {14325}}
    }
    child {node  (f) {123}
        child {node  (g) {52341}}
        child {node  (h) {25314}}
        child {node  (i) {21354}}
        child {node  (j) {12345}}
    };

\end{tikzpicture}
\caption{Generating tree for Baxter permutations of odd length fixed under conjugation by the longest element}
\end{figure}
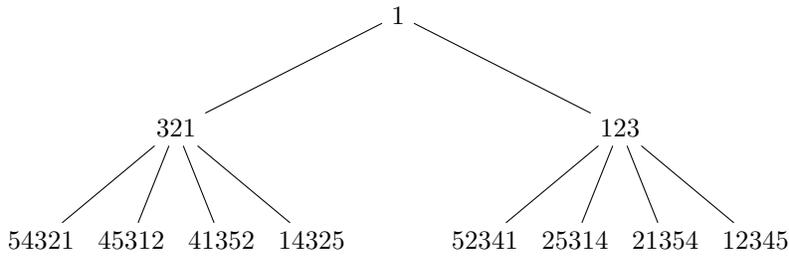

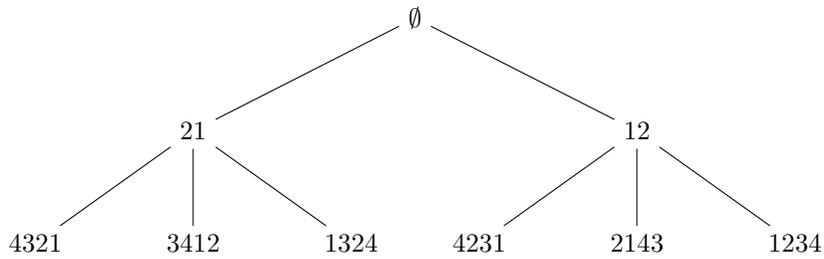
\begin{figure}
\begin{tikzpicture}
\tikzstyle{level 1}=[sibling distance=59mm]
\tikzstyle{level 2}=[sibling distance=21mm]
\tikzstyle{level 3}=[sibling distance=5mm]
\node  (z){$\emptyset$}
    child {node  (a) {21}
        child {node  (b) {4321}}
        child {node (c) {3412}}
        child {node (d) {1324}}
    }
    child {node  (f) {12}
        child {node  (g) {4231}}
        child {node  (h) {2143}}
        child {node  (i) {1234}}
    };
\end{tikzpicture}
\caption{Generating tree for Baxter permutations of even length fixed under $180^{\circ}$ rotation of permutation matrix.}
\end{figure}

We already have a combinatorial rule for when we can insert a new largest entry into a Baxter permutation and still be a Baxter permutation, so now we come up with a combinatorial rule for when we can insert a new smallest entry into a Baxter permutation and still be a Baxter permutation. By inserting a new smallest entry, we mean that we increase all the labels in the existing permutation by 1, and then insert a new entry with label 1, so that if the original permutation was a standard permutation of $n$ letters on $[n]$, then the result will be a standard permutation on $[n+1]$.

\begin{lemma}
\label{smallinsertion}
Inserting a new smallest label into position $j$ into a permutation is equivalent to rotating the permutation matrix $180^{\circ}$, inserting $n$ into position $n+1-j$, and then rotating the permutation matrix $180^{\circ}$ again.

Consequently, given a Baxter permutation $w$, the admissible places we can insert a new smallest label are immediately to the left of a left-to-right minimum, or immediately to the right of a right-to-left minimum.
\end{lemma}

\begin{proof}
We use the fact that $w$ is a Baxter permutation if and only if it is a Baxter permutation when we reverse the labels (i.e., send $i$ to $n+1-i$). Then if we take $w$, reverse the labels, insert a new largest label into position $i$, and then reverse the labels again, this is equivalent to inserting 1 into position $i$. Thus, we can insert 1 into position $i$ if and only if we can insert $n+1$ into position $i$ in the reversed word. This happens if and only if position $i$ is immediately to the left of a left-to-right maximum, or immediately to the right of a right-to-left maximum in the reversed word. This happens if and only if position $i$ is immediately to the left of a left-to-right minimum, or immediately to the right of a right-to-left minimum in the original word.
\end{proof}

Now, we want to make the generating tree using this insertion rule.

\begin{theorem}
The generating tree for Baxter Permutations fixed under $180^{\circ}$ rotation of even (resp. odd) length is isomorphic to the tree with root $(0,0)$ (resp. $(1,1)$), and branching rule given by $(i,j)$ having children $(1,j+2), (2, j+1), \ldots (i,j+1), (i+1,j),\ldots (i+1,2), (i+2,1)$.

\end{theorem}

\begin{proof}

The proof of Lemma~\ref{smallinsertion} makes it clear that if we can insert $n$ into position $i$ of $w$, then we can insert $1$ into position $n+1-i$ of $w_0ww_0$. So if $w$ is fixed under conjugation by the longest element, then we can insert $n+1$ into a position if and only if we can insert $1$ into the complementary position. However, we need to check to make sure that we can still insert $1$ into a complementary position after we have inserted $n+1$.

If $i$ is less than $n/2$, then the complementary place we want to insert 1 will be shifted right by 1. If $i$ is greater than $n/2$, then the complementary place we want to insert 1 will still be $n+1-i$. In either of these cases, the act of inserting $n$ will not affect being able to insert 1 into the complementary position, because the combinatorial rule for inserting 1 depends on things being left-to-right and right-to-left minima, and inserting a new largest label will not affect that.

As a kind of boundary case, we have the situation where $n$ is even and $i=n/2$, so we are inserting $n+1$ into the middle of the word. Then we could insert 1 either immediately to the left or right of $n+1$ and still have a permutation fixed under conjugation by the longest element. However, exactly one of these choices will correspond to a Baxter permutation. 

Without loss of generality, say we inserted $n+1$ to the right of a left-to-right maxima, $w_{n/2}$. Then $w_{n/2+1}$ will be a right-to-left minima, even after we insert $n+1$. So we can insert 1 to the left of $w_{n/2+1}$, which will be immediately to the right of $n+1$. This means that $w_{n/2}>w_{n/2+1}$, or else the subword $w_{n/2}(n+1)1w_{n/2+1}$ would be a copy of the vincular pattern $2-41-3$. Thus, if we inserted 1 on the other side of $n+1$, we would be making a subword that was an instance of the forbidden vincular pattern $3-14-2$.

Now, we want to make an isomorphic generating tree that doesn't require us to keep track of the permutation in full, analogous to what we did with all Baxter permutations.

Again, we only need to keep track of the number of left-to-right and right-to-left maxima. Each of these corresponds to a place where we can insert $n+1$, and then we know there will be a complementary place we can insert 1 to stay fixed under conjugation by the longest element. We know how inserting $n+1$ will affect the number of left-to-right and right-to-left maxima. Inserting 1 will in general not create any new left-to-right or right-to-left maxima, except in the case where we are adding 1 to the beginning or end of the word. Thus, we get the desired branching rule (see Figure~\ref{fig:conjbranch}).
\end{proof}

\begin{figure}
\begin{tikzpicture}
\tikzstyle{level 1}=[sibling distance=15mm]
\node (A) {$(i,j)$}
 [level distance=30mm]
 child {node {$(1,j\!+\!2)$}}
 child {node {$(2,j\!+\!1)$}}
 child {node {$\ldots$}}
 child {node {$(i,j\!+\!1)$}}
 child {node {$(i\!+\!1,j)$}}
 child {node {$(i\!+\!1,j\!-\!1)$}}
 child {node {$\ldots$}}
 child {node {$(i\!+\!2,1)$}}
;
\end{tikzpicture}
\caption{Rule for generating tree isomorphic to Baxter permutations fixed under $180^{\circ}$ rotation of permutation matrix.}
\label{fig:conjbranch}
\end{figure}
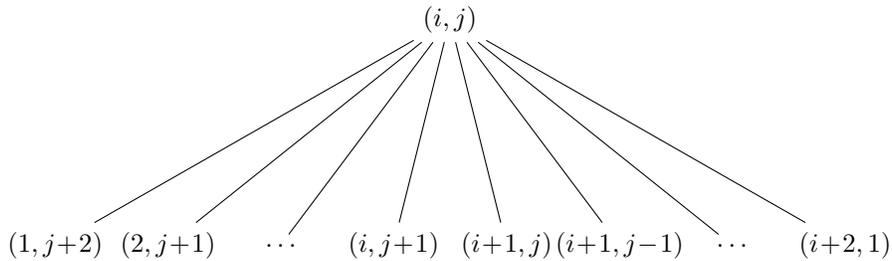

In principle, one could try and analyze this branching rule to come up with an algebraic formula for the number of Baxter permutations fixed under conjugation by the longest element. While it may give a refined enumeration for the number of Baxter permutations fixed under conjugation by the longest element with a given number of left-to-right and right-to-left maxima, it is unlikely that the resulting expression for the entire set would be as elegant as the $``q=-1''$ formula in \cite{Dilks}. In practice, this is more of a stepping stone to the case of Baxter permutations fixed under $90^{\circ}$ rotation, where the rules for insertion are more technical, but the resulting branching structure has a transparent enumerative formula.

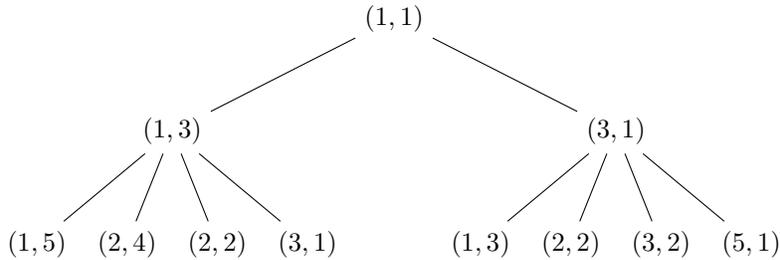
\begin{figure}
\begin{center}
\begin{tikzpicture}
\tikzstyle{level 1}=[sibling distance=59mm]
\tikzstyle{level 2}=[sibling distance=12mm]
\tikzstyle{level 3}=[sibling distance=5mm]
\node  (z){$(1,1)$}
    child {node  (a) {$(1,3)$}
        child {node  (b) {$(1,5)$}}
        child {node (c) {$(2,4)$}}
        child {node (d) {$(2,2)$}}
        child {node (e) {$(3,1)$}}
    }
    child {node  (f) {$(3,1)$}
        child {node  (g) {$(1,3)$}}
        child {node  (h) {$(2,2)$}}
        child {node  (i) {$(3,2)$}}
        child {node  (j) {$(5,1)$}}
    };

\end{tikzpicture}
\caption{Isomorphic generating tree for Baxter permutations of odd length fixed under $180^{\circ}$ rotation of permutation matrix.}
\end{center}
\end{figure}

\begin{figure}
\begin{tikzpicture}
\tikzstyle{level 1}=[sibling distance=59mm]
\tikzstyle{level 2}=[sibling distance=21mm]
\tikzstyle{level 3}=[sibling distance=5mm]
\node  (z){$(0,0)$}
    child {node  (a) {$(1,2)$}
        child {node  (b) {$(1,4)$}}
        child {node (c) {$(2,2)$}}
        child {node (d) {$(3,1)$}}
    }
    child {node  (f) {$(2,1)$}
        child {node  (g) {$(1,3)$}}
        child {node  (h) {$(2,2)$}}
        child {node  (i) {$(4,1)$}}
    };

\end{tikzpicture}
\caption{Isomorphic generating tree for Baxter permutations of even length fixed under conjugation by the longest element}
\end{figure}
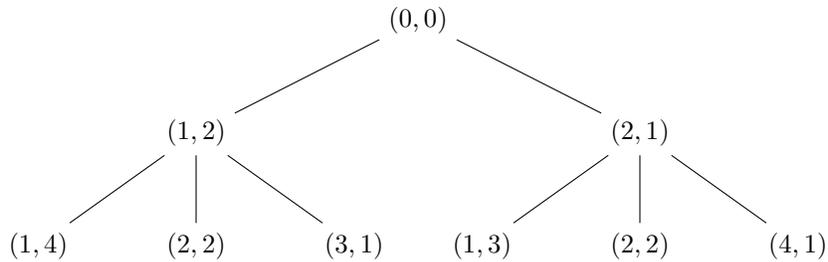

\section{Generating tree for Baxter permutations fixed under $90^{\circ}$ rotation.}
\label{sec:GenTreeQuarter}

Now, we generalize the approach we used for Baxter permutations fixed under $180^{\circ}$ rotation to those fixed under $90^{\circ}$ rotation.

First, we determine for which values of $n$ a Baxter permutation of length $n$ can possibly be fixed under $90^{\circ}$ rotation.

\begin{lemma}
If a permutation of length $n$ is fixed under $90^{\circ}$ rotation, then $n$ must be $4m$ or $4m+1$ for some positive integer $m$.
\end{lemma}

\begin{proof}
For a permutation of length $n$ to be fixed under $90^{\circ}$ rotation, it is equivalent to say that if $w_i=j$, then $w_j=n+1-i$, $w_{n+1-i}=n+1-j$, and $w_{n+1-j}=i$. If we consider the cycle structure of this permutation, in general it makes a 4-cycle $(i, j, n+1-i, n+1-j)$. If this were to degenerate into a smaller cycle, we would have that $i=n+1-i$. This forces to $n=2i+1$ to be odd, and it also forces $i=j$, which means it actually degenerates to a single central fixed point. Thus, a permutation fixed under this action must have length $4m$ or $4m+1$, either consisting solely of 4-cycles, or 4-cycles and a single central fixed point.
\end{proof}

\begin{lemma}
If $w$ is a Baxter permutation of length $n$ fixed under $90^{\circ}$ rotation, then $n$ must be odd.
\end{lemma}

\begin{proof}
If $w$ is a Baxter permutation of length $n$ with $k$ descents, then by Corollary~\ref{desisinvdes}, $w^{-1}$ will have $k$ descents, and $w_0w^{-1}$ will have $n-1-k$ descents. So for a Baxter permutation to be fixed by this action, we must have $k=n-1-k$, which implies that $n$ must be odd. 
\end{proof}

\begin{corollary}
\label{Length4PlusOne}
If $w$ is a Baxter permutation of length $n$, then $n=4m+1$ for some integer $m$.

In particular, a Baxter permutation fixed under $90^{\circ}$ rotation will consist of a single central fixed point, and 4-cycles of the form $(i, j, n+1-i, n+1-j)$.
\end{corollary}

For $n>1$, such a permutation will have a 4-cycle of the form $(1, j, n, n+1-j)$, which means the permutation starts with $j$, has $n$ in the $j^{th}$ position, 1 in the $(n+1-j)^{th}$ position, and $n+1-j$ at the end. We already know that we can remove $n$ and $1$ from a Baxter permutation and still be a Baxter permutation. It is not hard to see that we can also remove the first element or the last element from a Baxter permutation and still be a Baxter permutation (after reducing labels so we're still a permutation on $[n]$). So if we take a Baxter permutation fixed under $90^{\circ}$ rotation and the remove the largest label, the smallest label, the first label, and the last label, then we will still have a Baxter permutation, and it will still be fixed under $90^{\circ}$ rotation.

Thus, we can create a generating tree, with the identity permutation on 1 element as the root.

%

\begin{figure}
\begin{center}
\begin{tikzpicture}

\tikzstyle{level 1}=[sibling distance=30mm]
\tikzstyle{level 2}=[sibling distance=6mm]
\node  (z){1}
[grow=right, level distance=35mm]
    child { node  (a) {41352}
        child {node (b) {296357418}}
        child {node (c) {672159834}}
        child {node (d) {761258943}}
        child {node (e) {816357492}}
    }
    child { node  (f) {25314}
        child {node  (g) {294753618}}
        child {node  (h) {349852167}}
        child {node  (i) {438951276}}
        child {node  (j) {814753692}}
    };

\end{tikzpicture}
\end{center}
\caption{Start of generating tree for Baxter permutations fixed under $90^{\circ}$ rotation (drawn left-to-right).}
\end{figure}
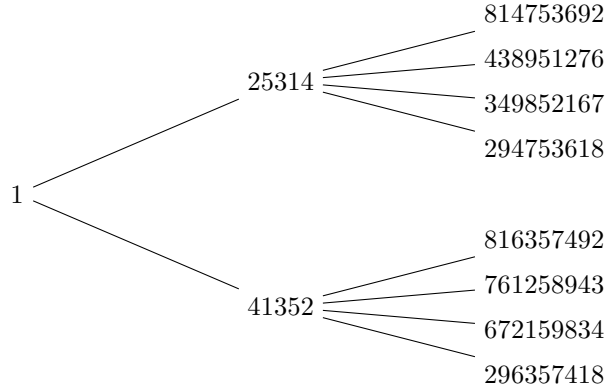

In order to create a 4-cycle, we have to come up with a combinatorial rule for when we can insert a letter at the beginning (resp. end) of a Baxter permutation, and still have it be a Baxter permutation. To insert a letter $j$ at the beginning of a permutation $w$ of length $n$, we mean that we increase all the labels greater than or equal to $j$ in $w$ by 1, and then prepend $j$, so the result is a standard permutation on $[n+1]$.

\begin{lemma}
\label{endinsertion}
Inserting $j$ at the end of a word is equivalent to rotating the permutation matrix $90^{\circ}$ clockwise, inserting $n$ into position $n+1-j$, and the rotating the permutation matrix $90^{\circ}$ counter-clockwise.

Similarly, inserting $j$ at the beginning of a word is equivalent to rotating a permutation matrix $90^{\circ}$ counter-clockwise, inserting $n$ into position $j$, and then rotating back $90^{\circ}$ clockwise.

Consequently, we can insert $j$ at the end (resp. beginning) of a Baxter permutation and still have it be a Baxter permutation if and only if all entries smaller than $j$ appear to the left (resp. right) of $j$, or if all entries bigger than $j-1$ appear to the right (resp. left) of $j-1$.
\end{lemma}

\begin{proof}
Before, we thought of inserting 1 as reversing the labels of a permutation, inserting $n$, and reversing the word again. In this case, we can see that inserting $j$ at the end of a word is the same as taking the inverse of a permutation, inserting $n$ into the $j^{th}$ position, and then taking the inverse of the resulting permutation.

So we can append $j$ to a Baxter permutation and still be a Baxter permutation if and only if we can insert $n$ into position $j$ in $w^{-1}$. We can do this if and only if position $j$ is immediately to the left of a left-to-right maximum, or immediately to the right of a right-to-left maximum, which means that either $w^{-1}_j>w^{-1}_i$ for all $i<j$, or $w^{-1}_{j-1}>w_i$ for all $i>j-1$.

Again by symmetry, we can insert $j$ at the beginning of a Baxter permutation $w$ if and only if we can insert $j$ at the end of $w$ reversed, which leads to a similar combinatorial description with left and right interchanged.
\end{proof}

Note that inserting $j$ at the end (resp. beginning) of Baxter permutation can possibly decrease the number of right-to-left (resp. left-to-right) maxima, as any previous left-to-right (resp. right-to-left) maxima that was less than $j$ will no longer be one after $j$ is inserted at the end (resp. beginning).

\begin{theorem}
\label{QuadInsertion} 
For a Baxter permutation fixed under $90^{\circ}$ rotation, for every admissible position we can insert a new largest label and still have a Baxter permutation, it is also possible to insert a new smallest label, a new beginning label, and a new final label so that the result is a Baxter permutation fixed under $90^{\circ}$ rotation.
\end{theorem}

\begin{proof}
Without loss of generality, assume we are inserting $n$ to the right of a right-to-left maxima. The procedure for when we can insert $n+1$ to the left of a left-to-right maxima is the same, except we reverse the order of the word, follow the procedure for inserting $n+1$ to the right of a right-to-left maxima, and then reverse the order of the resulting word.

Say $w$ is a Baxter permutation of length $n$ fixed under $90^{\circ}$ rotation, with a right-to-left maxima at $w_j$. This means that we could insert $n+1$ into position $j+1$, and by Lemma~\ref{smallinsertion} we could also insert 1 into position $n-j$, and by Lemma~\ref{endinsertion} we could insert $n-j$ at the end or $j+1$ at the beginning. Specifically, since we know that we'd be inserting $n+1$ to the right of a left-to-right maxima, we know that $w_{n+j-1}$ must be a right-to-left minima, and that all entries larger than $j$ appear to the left of $j$, and that all entries smaller than $n-j$ appear to the left of $n-j$. We also know that as a left-to-right maxima, $w_j$ must be at least $n-j$, since there are $n-j$ things to its right that must be smaller.

But we need to check that we can perform all four insertions sequentially in a way so that the result is a Baxter permutation fixed under $90^{\circ}$ rotation with a new 4-cycle added.

We will seperately consider the cases with $j+1<n/2$, and $j+1>n/2$. Note that since $n$ has to be odd, we don't have to deal with the special case of $j+1=n/2$.

First, suppose $j+1<n/2$. We insert $j+1$ at the beginning first, which increases any labels that were $j+1$ or higher by 1. So now, we want to insert $n+1-j$ at the end. We have to check that all labels less than $n+1-j$ are to the left of $n+1-j$. Since in the original permutation we had that all labels less than $n-j$ were to the left of $n-j$, when we add 1 to all labels $j+1$ or higher, we will have that all labels less than $n+1-j$ (except possibly $j+1$) are to the left of $n+1-j$. But since we add $j+1$ to the beginning of the word, it will also certainly be to the left of $n+1-j$. So we may insert $(n+1-j)$ at the end.

We now have a permutation of length $n+2$, which will still be fixed under $180^{\circ}$ rotation. So by the previous section, if we can insert a new largest label into some position, we know we can insert a new smallest label into the complementary position. The $(j+1)^{st}$ entry in this permutation will be $w_j+2$, as inserting two smaller labels increased its label by 2, and inserting a label at the beginning shifted it right by one. We need to check that this is still a right-to-left maxima. The only thing we did that could have changed this is inserting $n+1-j$ at the end. However, since $w_j\geq n-j$, we have $w_j+2\geq n+2-j$, so adding $n+1-j$ will not keep it from being a right-to-left-maxima. Thus, we can insert a new largest label into position $j+2$, and also a new smallest label into the complementary position $(n+1-j)$.

After all of these steps, we will now have $j+2$ in the first position, $n+4$ in position $j+2$, 1 in position $(n+4)-(j+2)$, and $(n+2-j)$ at the end, which creates the desired 4-cycle.

Now, suppose $j+1>n/2$. Again, we insert $j+1$ at the beginning. Now, we want to insert $n-j$ at the end. Inserting $j+1$ will not affect any labels $n-j$ or smaller, so we will still have that all labels less than $n-j$ are to the left of $n-j$.

Again, we have a permutation of length $n+2$ fixed under $180^{\circ}$ rotation, so it suffices to show we can place a new largest label, and it will automatically follow that a new smallest label can go in the complementary position. Consider the $(j+1)^{st}$ entry of this permutation, which was originally $w_j$. We claim this is still a right-to-left maxima. The only thing that could have changed this fact is inserting $n-j$ at the end. Since $w_j\geq n-j$, this label would at least be increased by 1 when we inserted $n-j$. It could possibly also be increased by 1 when we inserted $j+1$, but what's important is that the $(j+1)^{st}$ entry is at least $n+1-j$, and thus having $n-j$ at the end will not prevent it from being a left-to-right maxima.

After all of these steps, we will now have $j+3$ in the first position, $n+4$ in position $j+3$, 1 in position $(n+4)-(j+3)$, and $n+1-j$ at the end, which creates the desired 4-cycle.
\end{proof}

Now, we want analyze how doing these four insertions changes the number of left-to-right and right-to-left maxima.

\begin{lemma}
\label{LeftIsRight}
If $w$ is a Baxter permutation fixed under $90^{\circ}$ rotation, then $w$ has the same number of left-to-right and right-to-left maxima. In particular, if $w$ has left-to-right maxima in positions $x_1<x_2<\ldots <x_j$ and right-to-left maxima at positions $y_j<y_{j-1}<\ldots<y_1$, and we do a 4-cycle insertion corresponding to being able to insert a new largest label to the right of $w_{y_i}$(or to the left of $w_{x_i}$), then the resulting Baxter permutation fixed under $90^{\circ}$ rotation will have $i+1$ left-to-right maxima and $i+1$ right-to-left maxima.
\end{lemma}

\begin{proof}
We note that $w_j<w_i$ if and only if $j$ appears to the left of $i$ in $w^{-1}$ if and only if $i$ appears to the left of $j$ in $w_0w^{-1}$. Thus, if $w=w_0w^{-1}$, we have a right-to-left maxima in position $j$ if and only if $j$ is a left-to-right maxima (and similarly, a left-to-right maxima in position $j$ if and only $j$ is a right-to-left maxima). So if $w$ is fixed by this action, it must have the same number of right-to-left maxima as left-to-right maxima.

Additionally, this gives a bijection between right-to-left maxima that were originally in $w$ that are later killed by $n+3$, and left-to-right maxima that are killed by the $j+2$ or $j+3$ at the beginning of the word. Similarly, there is a bijection between right-to-left maxima originally in $w$ that are later killed by the final entry, and left-to-right maxima originally in $w$ later killed by 1.

Since 1 always ends up on the interior of the word, it will never be a left-to-right maxima, and so the final entry will also never kill anything that was originally a right-to-left maxima. Since we (WLOG) did an insertion corresponding to putting a new largest label to the right of $w_{y_i}$, $n+3$ will kill the right-to-left maxima $w_{y_{i}},\ldots w_{y_{j}}$. Thus, we will have $i+1$ right-to-left maxima; the new right-most entry, the $i-1$ original right-to-left maxima not killed by $n+3$, and $n+3$.
\end{proof}

We now have enough information to analyze the generating tree for Baxter permutations fixed under rotation by $90^{\circ}$. If a Baxter permutation fixed under rotation by $90^{\circ}$ has $i+1$ left-to-right maxima and $i+1$ right-to-left maxima, then it will have $2i+2$ children. There will be $i+1$ children with number of left-to-right (and right-to-left) maxima being $2,3,\ldots i+2$ corresponding to inserting a new largest label to the left of a left-to-right maxima, and $i+1$ children with number of left-to-right (and right-to-left) maxima being $2,3,\ldots i+2$ corresponding to inserting a new largest label to the right of a right-to-left maxima.

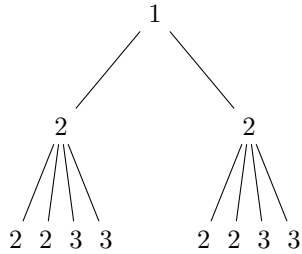
\begin{figure}
\begin{center}
\begin{tikzpicture}
\tikzstyle{level 1}=[sibling distance=25mm]
\tikzstyle{level 2}=[sibling distance=4mm]
\node  (z){1}
    child {node  (a) {2}	
		child {node (a1) {2}}
		child {node (a2) {2}}
		child {node (a3) {3}}
		child {node (a4) {3}}
		}
		child { node (b) {2}
		child {node (b1) {2}}
		child {node (b2) {2}}
		child {node (b3) {3}}
		child {node (b4) {3}}
		};
	\end{tikzpicture}
\end{center}
\caption{The beginning of the doubled Catalan tree}
\label{doublecattree}
\end{figure}

We may now prove the main result.

\begin{MainResult}
The number of Baxter permutations of length $n$ fixed under $90^{\circ}$ rotation of its permutation matrix is $2^mC_m$ (where $C_m$ is the Catalan number) if $n=4m+1$, and zero otherwise.
\end{MainResult}

\begin{proof}
By Corollary~\ref{Length4PlusOne}, $n$ must be $4m+1$ for some integer $m$.

Now, we consider the generating tree on Baxter permutations fixed under $90^{\circ}$ rotation. By Lemma~\ref{QuadInsertion}, if a Baxter permutation fixed under rotation by $90^{\circ}$ has $i+1$ left-to-right maxima and $i+1$ right-to-left maxima, then it will have $2i+2$ children. By Lemma~\ref{LeftIsRight}, there will be $i+1$ children with number of left-to-right (and right-to-left) maxima being $2,3,\ldots i+2$ corresponding to inserting a new largest label to the left of a left-to-right maxima, and $i+1$ children with number of left-to-right (and right-to-left) maxima being $2,3,\ldots i+2$ corresponding to inserting a new largest label to the right of a right-to-left maxima. 

Thus, we can identify a Baxter permutation fixed under $90^{\circ}$ rotation with $i+1$ left-to-right maxima (and thus $i+1$ right-to-left maxima) with the number $i+1$, and its descendents will be identified with the numbers $2,2,3,3,\ldots i+2,i+2$.

This generating tree is almost like the Catalan tree, except each parent with label $i+1$ has two (not one) children with a label between $2$ and $i+2$, and our root will have label 1. This implies that the number of elements of a given rank $m$ must be $2^mC_m$. See Figure~\ref{doublecattree}.
\end{proof}

\section{Remarks}

The fact that this enumeration has such an elegant closed formula means that it is likely that there is an underlying combinatorial bijection. However, as with Chung, Graham, Hoggat, and Kleiman, the method of generating trees does not make such an interpretation transparent.

Additionally, one might hope that it is possible to extend the previous ``q=-1'' result for Baxter permutations fixed under $180^{\circ}$ rotation to an instance of the cyclic sieving phenomenon. That is to say, finding a polynomial $f(q)$ where gives an enumeration of Baxter permutations (perhaps with respect to some statistics), $f(-1)$ counts how many of these Baxter permutations are fixed under $180^{\circ}$ rotation, and $f(i)=f(-i)$ counts how many of them are fixed under $90^{\circ}$ rotation. However, the natural candidate used in ~\cite{Dilks} does not give the right enumeration when evaluated at $i$, and it does not appear that it can be easily modified to give such a result.

\bibliographystyle{abbrv}
\bibliography{BaxterRotation}
\label{sec:biblio}

\end{document}